\documentclass[11pt]{article}
\usepackage{amssymb}
\usepackage{amsmath}
\usepackage{amsfonts}

\providecommand{\U}[1]{\protect \rule{.1in}{.1in}}
\newtheorem{theorem}{Theorem}

\newtheorem{corollary}[theorem]{Corollary}

\newtheorem{lemma}[theorem]{Lemma}

\newtheorem{proposition}[theorem]{Proposition}
\newtheorem{remark}[theorem]{Remark}

\newenvironment{proof}[1][Proof]{\noindent \textbf{#1.} }{\  \rule[]{0.5em}{0.5em}}
\include {mak}
\input{tcilatex}
\begin{document}

\begin{center}
{\LARGE Identities and congruences involving the Fubini polynomials}

\  \  \  \ 

{\large Miloud Mihoubi\footnotemark[1] \ and \ Said Taharbouchet%
\footnotemark[2]}

USTHB, Faculty of Mathematics, RECITS Laboratory, PB 32 El Alia 16111
Algiers, Algeria.

{\large \footnotemark[1]}mmihoubi@usthb.dz \ {\large \footnotemark[2]}%
said.taharbouchet@gmail.com

{\large \  \  \ }
\end{center}

\noindent \textbf{Abstract. }In this paper, we investigate the umbral
representation of the Fubini polynomials $\mathbf{F}_{\mathbf{x}}^{n}:=%
\mathcal{F}_{n}\left( x\right) $ to derive some properties involving these
polynomials. For any prime number $p$ and any polynomial $f$ with integer
coefficients, we show $\left( f\left( \mathbf{F}_{\mathbf{x}}\right) \right)
^{p}\equiv f\left( \mathbf{F}_{\mathbf{x}}\right) $ and we give other
curious congruences.

\noindent \textbf{Keywords.} Fubini umbra, Fubini polynomials, identities,
congruences.

\noindent \textbf{2000 MSC:} 05A18, 05A40, 11A07.

\section{Introduction}

The Fubini numbers are quantities arising from enumerative combinatorics and
have nice number-theoretic properties. In combinatorics, the $n$-th Fubini
number $\mathcal{F}_{n}$ (named also the $n$-th ordered Bell number) counts
the number of ways to partition the set $\left[ n\right] :=\left \{ 1,\ldots
,n\right \} $ into ordered subsets \cite{boy,mah}. The Fubini polynomials are
defined by $\mathcal{F}_{n}\left( x\right) =\underset{k=0}{\overset{n}{\sum }%
}{n\brace k}k!x^{k}$ and satisfy the recurrence relation $\left( x+1\right) 
\mathcal{F}_{n}\left( x\right) =x\underset{j=0}{\overset{n}{\sum }}\binom{n}{%
j}\mathcal{F}_{j}\left( x\right) ,\ n\geq 1,$ where ${n\brace k}$ is the $%
\left( n,k\right) $-th Stirling number of the second kind \cite{boy,sta}.
For $x=1$ we obtain the Fubini numbers $\mathcal{F}_{n}=\underset{k=0}{%
\overset{n}{\sum }}{n\brace k}k!$ \cite{das,dum,fla,jam,tan,mah,vel}.\newline
More generally, let $\mathcal{F}_{n}\left( x,r,s\right) $ be the $n$-th $%
\left( r,s\right) $-Fubini polynomial defined by%
\begin{equation*}
\mathcal{F}_{n}\left( x,r,s\right) =\underset{k=0}{\overset{n}{\sum }}{n+r%
\brace k+r}_{r}\left( k+s\right) !x^{k}.
\end{equation*}%
This polynomial generalizes the Fubini polynomial $\mathcal{F}_{n}\left(
x\right) =\mathcal{F}_{n}\left( x;0,0\right) $ and the $r$-Fubini polynomial 
$\mathcal{F}_{n,r}\left( x\right) =\mathcal{F}_{n}\left( x;r,r\right) $
introduced by Mez\H{o} \cite{mez}. Here, ${n\brace k}_{r}$ denotes the $%
\left( n,k\right) $-th $r$-Stirling number of the second kind \cite{bro}.
One can see easily that%
\begin{align*}
\mathcal{F}_{0}\left( x,r,s\right) & =s!, \\
\mathcal{F}_{1}\left( x,r,s\right) & =s!\left( r+\left( s+1\right) x\right) ,
\\
\mathcal{F}_{2}\left( x,r,s\right) & =s!\left( r^{2}+\left( 2r+1\right)
\left( s+1\right) x+\left( s+1\right) \left( s+2\right) x^{2}\right) .
\end{align*}%
As it shown below, these polynomials are also linked to the absolute $r$%
-Stirling numbers of first kind denoted by ${n\brack k}_{r}.$ Recall that
the $r$-Stirling numbers can be defined by \cite{bro,sta} 
\begin{equation*}
\left( x\right) _{n}=\underset{k=0}{\overset{n}{\sum }}\left( -1\right)
^{n-k}{n+r\brack k+r}_{r}\left( x+r\right) ^{k}\text{ and }\left( x+r\right)
^{n}=\underset{k=0}{\overset{n}{\sum }}{n+r\brace k+r}_{r}\left( x\right)
_{k},
\end{equation*}%
where $\left( \alpha \right) _{n}=\alpha \cdots \left( \alpha -n+1\right) $\
if $n\geq 1$, $\left( \alpha \right) _{0}=1,$ ${n\brack k}={n\brack k}_{0}$
and ${n\brace k}={n\brace k}_{0},$ \newline
\newline
This work is motivated by application of the umbral calculus method to
determine identities and congruences involving Bell numbers and polynomials
in the works of Gessel \cite{ges}, Sun et al. \cite{sun} and Benyattou et
al. \cite{ben}. In this paper, we will talk about identities and congruences
involving the $\left( r,s\right) $-Fubini polynomials based on the Fubini
umbra defined by $\mathbf{F}_{\mathbf{x}}^{n}:=\mathcal{F}_{n}\left(
x\right) .$ \newline
For more information about umbral calculus, see \cite%
{di,ges,rob,rom,rom1,rota}.

\section{Identities for the $\left( r,s\right) $-Fubini polynomials}

By definition of the Fubini umbra, it follows that the above recurrence
relation can be rewritten as $\left( x+1\right) \mathbf{F}_{\mathbf{x}%
}^{n}=x\left( \mathbf{F}_{\mathbf{x}}+1\right) ^{n}, \ n\geq1.$ Furthermore,
we have

\begin{proposition}
\label{P0}Let $f$ be a polynomial and $r,s$ be non-negative integers. Then 
\begin{align*}
&\left( x+1\right) f\left( \mathbf{F}_{\mathbf{x}}+r\right) =xf\left( 
\mathbf{F}_{\mathbf{x}}+r+1\right)+f\left( r \right), \\
&\left( \mathbf{F}_{\mathbf{x}}+r\right) _{n+r}=\left(
n+r\right)!x^{n}\left( x+1\right) ^{r} \\
&\left( \mathbf{F}_{\mathbf{x}}+r-s\right) ^{n}\left( \mathbf{F}_{\mathbf{x}%
}\right) _{s} =x^{s}\mathcal{F}_{n}\left( x;r,s\right), \\
&\left( \mathbf{F}_{\mathbf{x}}+r\right) ^{n}\left( \mathbf{F}_{\mathbf{x}%
}+s\right) _{s} =\left( x+1\right) ^{s}\mathcal{F}_{n}\left( x;r,s\right) .
\end{align*}
\end{proposition}

\begin{proof}
It suffices to show the first identity for $f\left( x\right) =x^{n}$. For $%
r=0\ $we have $\left( x+1\right) \mathbf{F}_{\mathbf{x}}^{n}-x\left( \mathbf{%
F}_{\mathbf{x}}+1\right) ^{n}=\delta _{\left( n=0\right) }.$ Assume it is
true for $r-1,$ then if we set 
\begin{equation*}
h_{n}(r):=\left( x+1\right) \left( \mathbf{F}_{\mathbf{x}}+r\right)
^{n}-x\left( \mathbf{F}_{\mathbf{x}}+r+1\right) ^{n}
\end{equation*}%
we obtain $h_{n}(r)=\underset{j=0}{\overset{n}{\sum }}\binom{n}{j}h_{j}(r-1)=%
\underset{j=0}{\overset{n}{\sum }}\binom{n}{j}\left( r-1\right) ^{j}=r^{n},$
which concludes the induction step. For the other identities, since $\left(
x\right) _{n}=\underset{k=0}{\overset{n}{\sum }}\left( -1\right) ^{n-k}{n%
\brack k}x^{k}$ and $\left( x\right) _{n}$ is a sequence of binomial type 
\cite{rom,mih}, we \ obtain%
\begin{equation*}
\left( \mathbf{F}_{\mathbf{x}}+r\right) _{n+r}=\underset{j=0}{\overset{n+r}{%
\sum }}\binom{n+r}{j}\left( r\right) _{j}\left( \mathbf{F}_{\mathbf{x}%
}\right) _{n+r-j}=\left( n+r\right) !x^{n}\left( x+1\right) ^{r}.
\end{equation*}%
So, the polynomials $x^{s}\mathcal{F}_{n}\left( x;r,s\right) $ and $\left(
x+1\right) ^{s}\mathcal{F}_{n}\left( x,r,s\right) $ must be, respectively, 
\begin{align*}
\underset{j=0}{\overset{n}{\sum }}{n+r\brace j+r}_{r}\left( \mathbf{F}_{%
\mathbf{x}}\right) _{j+s}& =\underset{j=0}{\overset{n}{\sum }}{n+r\brace j+r}%
_{r}\left( \mathbf{F}_{\mathbf{x}}-s\right) _{j}\left( \mathbf{F}_{\mathbf{x}%
}\right) _{s}=\left( \mathbf{F}_{\mathbf{x}}+r-s\right) ^{n}\left( \mathbf{F}%
_{\mathbf{x}}\right) _{s}, \\
\underset{j=0}{\overset{n}{\sum }}{n+r\brace j+r}_{r}\left( \mathbf{F}_{%
\mathbf{x}}+s\right) _{j+s}& =\underset{j=0}{\overset{n}{\sum }}{n+r\brace %
j+r}_{r}\left( \mathbf{F}_{\mathbf{x}}\right) _{j}\left( \mathbf{F}_{\mathbf{%
x}}+s\right) _{s}=\left( \mathbf{F}_{\mathbf{x}}+r\right) ^{n}\left( \mathbf{%
F}_{\mathbf{x}}+s\right) _{s}.
\end{align*}
\end{proof}

\noindent The the last two identities of Proposition \ref{P0} lead to

\begin{corollary}
\label{P11}For any polynomial $f$ and any non-negative integers $r,s$ we have%
\begin{equation*}
\left( x+1\right) ^{s}f\left( \mathbf{F}_{\mathbf{x}}+r-s\right) \left( 
\mathbf{F}_{\mathbf{x}}\right) _{s}=x^{s}f\left( \mathbf{F}_{\mathbf{x}%
}+r\right) \left( \mathbf{F}_{\mathbf{x}}+s\right) _{s}.
\end{equation*}
\end{corollary}

\begin{proposition}
\label{R1}Let $\mathcal{P}_{n}$ and $\mathcal{T}_{n}$ be the polynomials 
\begin{equation*}
\mathcal{P}_{n}\left( x;r\right) =\underset{j=0}{\overset{n}{\sum }}\left(
-1\right) ^{j}\binom{j+r}{r}x^{n-j}\text{ \ and \ }\mathcal{T}_{n}\left(
x;r\right) =\underset{j=0}{\overset{n}{\sum }}\binom{n+r}{j+r}x^{j}.
\end{equation*}%
Then \  \ $\left( \mathbf{F}_{\mathbf{x}}-r-1\right) _{n} = n!\mathcal{P}%
_{n}\left( x;r\right) $ \  \ and \  \ $\left( \mathbf{F}_{\mathbf{x}%
}+n+r\right) _{n} = n!\mathcal{T}_{n}\left( x;r\right) .$
\end{proposition}

\begin{proof}
It suffices to observe that 
\begin{eqnarray*}
\left( \mathbf{F}_{\mathbf{x}}-r-1\right) _{n} &=&\underset{j=0}{\overset{n}{%
\sum }}\binom{n}{j}\left( -r-1\right) _{j}\left( \mathbf{F}_{\mathbf{x}%
}\right) _{n-j}=n!\underset{j=0}{\overset{n}{\sum }}\left( -1\right) ^{j}%
\binom{j+r}{r}x^{n-j}, \\
\left( \mathbf{F}_{\mathbf{x}}+n+r\right) _{n} &=&\underset{j=0}{\overset{n}{%
\sum }}\binom{n}{j}\left( n+r\right) _{n-j}\left( \mathbf{F}_{\mathbf{x}%
}\right) _{j}=n!\underset{j=0}{\overset{n}{\sum }}\binom{n+r}{j+r}x^{j}. \  \
\  \  \  \  \  \  \  \  \  \ 
\end{eqnarray*}
\end{proof}

\noindent The identities of the following two theorems depend on the choice
of a polynomial $f$ and can be served to derive several identities and
congruences for the $\left( r,s\right) $-Fubini polynomials.

\begin{theorem}
\label{P12}Let $f$ be a polynomial and let $m,s$ be non-negative integers.
Then 
\begin{equation*}
\left( x+1\right) ^{m}f\left( \mathbf{F}_{\mathbf{x}}\right) -x^{m}f\left( 
\mathbf{F}_{\mathbf{x}}+m\right) =\underset{k=0}{\overset{m-1}{\sum}}f\left(
k\right) \left( x+1\right) ^{m-1-k}x^{k}, \  \ m\geq1.
\end{equation*}
\end{theorem}

\begin{proof}
Set $f\left( x\right) =\underset{k\geq 0}{\sum }a_{k}x^{k}$ and use the
first identity of Proposition \ref{P0} to obtain%
\begin{equation*}
\left( x+1\right) f\left( \mathbf{F}_{\mathbf{x}}\right) -xf\left( \mathbf{F}%
_{\mathbf{x}}+1\right) =f\left( 0\right) +\underset{k\geq 1}{\sum }%
a_{k}\left( \left( x+1\right) \mathbf{F}_{\mathbf{x}}^{k}-x\left( \mathbf{F}%
_{\mathbf{x}}+1\right) ^{k}\right) =f\left( 0\right) .
\end{equation*}%
So, the identity is true for $m=1.$ Assume it is true for $m.$ Then%
\begin{align*}
\left( x+1\right) ^{m+1}f\left( \mathbf{F}_{\mathbf{x}}\right) & =\left(
x+1\right) \left[ \underset{k=0}{\overset{m-1}{\sum }}\left( x+1\right)
^{m-1-k}x^{k}f\left( k\right) +x^{m}f\left( \mathbf{F}_{\mathbf{x}}+m\right) %
\right]  \\
& =\underset{k=0}{\overset{m-1}{\sum }}\left( x+1\right) ^{m-k}x^{k}f\left(
k\right) +x^{m}\left( \left( x+1\right) f\left( \mathbf{F}_{\mathbf{x}%
}+m\right) \right) 
\end{align*}%
and since $\left( x+1\right) f\left( \mathbf{F}_{\mathbf{x}}+m\right)
-xf\left( \mathbf{F}_{\mathbf{x}}+m+1\right) =f\left( m\right) ,$ we can
write%
\begin{align*}
\left( x+1\right) ^{m+1}f\left( \mathbf{F}_{\mathbf{x}}\right) & =\underset{%
k=0}{\overset{m-1}{\sum }}\left( x+1\right) ^{m-k}x^{k}f\left( k\right)
+x^{m}\left[ xf\left( \mathbf{F}_{\mathbf{x}}+m+1\right) +f\left( m\right) %
\right]  \\
& =\underset{k=0}{\overset{m-1}{\sum }}\left( x+1\right) ^{m-k}x^{k}f\left(
k\right) +x^{m}f\left( m\right) +x^{m+1}f\left( \mathbf{F}_{\mathbf{x}%
}+m+1\right)  \\
& =\underset{k=0}{\overset{m}{\sum }}\left( x+1\right) ^{m-k}x^{k}f\left(
k\right) +x^{m+1}f\left( \mathbf{F}_{\mathbf{x}}+m+1\right) 
\end{align*}%
which concludes the induction step.
\end{proof}

\begin{corollary}
\label{T}For any polynomial $f$ there holds 
\begin{equation*}
f(\mathbf{F}_{\mathbf{x}})=\frac{1}{1+x}{\displaystyle \sum \limits_{k\geq0}}
f(k)\left( \frac{x}{1+x}\right) ^{k},\  \ x>-\frac{1}{2}.
\end{equation*}
\end{corollary}

\begin{proof}
For $m=1$ in Theorem \ref{P12} we get when we replace $f(x)$ by $f(x+r)$: 
\newline
$f\left( r\right) =\left( x+1\right) f\left( \mathbf{F}_{\mathbf{x}%
}+r\right) -xf\left( \mathbf{F}_{\mathbf{x}}+r+1\right) .$ Then%
\begin{align*}
RHS& =\underset{n\rightarrow \infty }{\lim }\frac{1}{1+x}\underset{k=0}{%
\overset{n}{\sum }}\left( \frac{x}{1+x}\right) ^{k}\left( \left( x+1\right)
f(\mathbf{F}_{\mathbf{x}}+k)-xf(\mathbf{F}_{\mathbf{x}}+k+1)\right) \\
& =\underset{n\rightarrow \infty }{\lim }\left( f(\mathbf{F}_{\mathbf{x}%
})-\left( \frac{x}{1+x}\right) ^{n+1}f(\mathbf{F}_{\mathbf{x}}+n+1)\right)
=f(\mathbf{F}_{\mathbf{x}})
\end{align*}%
which completes the proof.
\end{proof}

\begin{corollary}
\label{C} Let $n,r,s$ be non-negative integers.\newline
For $f\left( x\right) =\left( x+r\right) ^{n}\left( x+s\right) _{s}$ or $%
\left( x+r-s\right) ^{n}\left( x\right) _{s}$ in Corollary \ref{T} we obtain%
\begin{equation*}
\mathcal{F}_{n}\left( x;r,s\right) =\frac{s!}{\left( 1+x\right) ^{s+1}}{%
\displaystyle \sum \limits_{k\geq0}} \binom{k+s}{s}(k+r)^{n}\left( \frac{x}{%
1+x}\right) ^{k}, \  \ x>-\frac{1}{2}.
\end{equation*}
\end{corollary}

\begin{corollary}
For any integers $r\geq0, \ s\geq0$ and $n\geq1$ the polynomial $\mathcal{F}%
_{n}\left( x,r,s+r\right)$ has only real non-positive roots.
\end{corollary}

\begin{proof}
From Corollary \ref{C} we may state 
\begin{equation*}
x^{r}\left( x+1\right) ^{s}\mathcal{F}_{n+1}\left( x;r,s+r\right) =x\frac{d}{%
dx}\left( x^{r}\left( x+1\right) ^{s+1}\mathcal{F}_{n}\left( x;r,s+r\right)
\right)
\end{equation*}%
and using the definition and the recurrence relation of $r$-Stirling numbers
we conclude that this property remains true for all real number $x.$ So, one
can verify by induction on $n$ that the polynomial $\mathcal{F}_{n}\left(
x;r,s+r\right) ,\ n\geq 1,$ has only real non-positive roots. \ 
\end{proof}

\begin{lemma}
For any non-negative integers $n\geq 2$ there holds%
\begin{equation*}
\left( 1+x\right) \mathcal{F}_{n-1}\left( x\right) =\underset{k=1}{\overset{n%
}{\sum }}{n \brace k}\left( k-1\right) !x^{k}.
\end{equation*}
\end{lemma}

\begin{proof}
From the definition of the Fubini polynomials, we have%
\begin{eqnarray*}
\left( 1+x\right) \mathcal{F}_{n-1}\left( x\right) &=&\underset{k=0}{\overset%
{n-1}{\sum }}{n-1\brace k}k!x^{k}+\underset{k=0}{\overset{n-1}{\sum }}{n-1%
\brace k}k!x^{k+1} \\
&=&\underset{k=1}{\overset{n}{\sum }}\left( k{n-1\brace k}+{n-1\brace k-1}%
\right) \left( k-1\right) !x^{k} \\
&=&\underset{k=1}{\overset{n}{\sum }}{n\brace k}\left( k-1\right) !x^{k}.\  \
\  \  \  \  \  \  \  \  \  \  \ 
\end{eqnarray*}
\end{proof}

\begin{proposition}
Let $n,r,s$ be non-negative integers. Then%
\begin{equation*}
\log \left(1+\underset{n\geq 1}{\sum }\frac{\mathcal{F}_{n}\left(
x;r,s\right)}{s!} \frac{t^{n}}{n!}\right)= \left(r+\left( s+1\right)
x\right)t+\left( s+1\right) \left( x+1\right) \underset{n\geq 2}{\sum }%
\mathcal{F}_{n-1}\left( x\right) \frac{t^{n}}{n!},
\end{equation*}%
In particular, for $r=s=0$ we get%
\begin{equation*}
\log \left(1+\underset{n\geq 1}{\sum } \mathcal{F}_{n}\left( x\right) \frac{%
t^{n}}{n!}\right)= xt+\left( x+1\right) \underset{n\geq 2}{\sum }\mathcal{F}%
_{n-1}\left( x\right) \frac{t^{n}}{n!} .
\end{equation*}
\end{proposition}

\begin{proof}
One can verify easily that the exponential generating function of the
polynomials $\mathcal{F}_{n}\left( x;r,s\right) $ is to be \ $s!\exp \left(
rt\right) \left( 1-x\left( \exp \left( t\right) -1\right) \right) ^{-s-1}.$
Then, upon using this generating function and the last Lemma, we can write 
\begin{eqnarray*}
LHS &=&rt-\left( s+1\right) \ln \left( 1-x\left( \exp \left( t\right)
-1\right) \right) \\
&=&rt+\left( s+1\right) \underset{k\geq 1}{\sum }\frac{x^{k}}{k}\left( \exp
\left( t\right) -1\right) ^{k} \\
&=&rt+\left( s+1\right) \underset{k\geq 1}{\sum }\left( k-1\right) !x^{k}%
\underset{n\geq k}{\sum }{n\brace k}\frac{t^{n}}{n!} \\
&=&rt+\left( s+1\right) xt+\left( s+1\right) \underset{n\geq 2}{\sum }\frac{%
t^{n}}{n!}\underset{k=1}{\overset{n}{\sum }}{n\brace k}\left( k-1\right)
!x^{k} \\
&=&\left( r+\left( s+1\right) x\right) t+\left( s+1\right) \left( x+1\right) 
\underset{n\geq 2}{\sum }\mathcal{F}_{n-1}\left( x\right) \frac{t^{n}}{n!}.
\  \  \  \  \  \  \  \  \  \  \  \  \ 
\end{eqnarray*}
\end{proof}

\section{Congruences on the (r,s)-Fubini polynomials}

In this section, we give some congruences involving the $(r,s)$-Fubini
polynomials. Let $\mathbb{Z}_{p}$ be the ring of $p$-adic integers and for
two polynomials $f\left( x\right) ,\ g\left( x\right) \in \mathbb{Z}_{p}%
\left[ x\right] ,$ the congruence $f\left( x\right) \equiv g\left( x\right)
\  \left( \func{mod}p\mathbb{Z}_{p}\left[x\right]\right) $ means that the
corresponding coefficients of $f\left( x\right) $ and $g\left( x\right) $
are congruent modulo $p.$ This congruence will be used later as $f\left(
x\right) \equiv g\left( x\right) $ and we will use $a \equiv b$ instead $a
\equiv b \  \left( \func{mod}p\right)$.

\begin{proposition}
\label{C3}Let $n,r,s$ be non-negative integers and $p$ be a prime number.
Then, for any polynomial $f$ with integer coefficients there holds%
\begin{equation*}
\underset{k=0}{\overset{p-1}{\sum}}f\left( k\right) \left( x+1\right)
^{p-1-k}x^{k}\equiv f\left( \mathbf{F}_{\mathbf{x}}\right) .
\end{equation*}
In particular, for $f\left( x\right) =\left( x+r-s\right) ^{n}\left(
x\right) _{s}$ or $\left( x+r\right) ^{n}\left( x+s\right) _{s}$ we get,
respectively,%
\begin{align*}
\underset{k=0}{\overset{p-1}{\sum}}\left( r-s+k\right) ^{n}\left( k\right)
_{s}\left( x+1\right) ^{p-1-k}x^{k} & \equiv x^{s}\mathcal{F}_{n}\left(
x;r,s\right) , \\
\underset{k=0}{\overset{p-1}{\sum}}\left( r+k\right) ^{n}\left( s+k\right)
_{s}\left( x+1\right) ^{p-1-k}x^{k} & \equiv \left( x+1\right) ^{s}\mathcal{F%
}_{n}\left( x;r,s\right) .
\end{align*}
\end{proposition}

\begin{proof}
For $m=p$ be a prime number, Theorem \ref{P12} implies 
\begin{equation*}
LHS=\left( x+1\right) ^{p}f\left( \mathbf{F}_{\mathbf{x}}\right)
-x^{p}f\left( \mathbf{F}_{\mathbf{x}}+p\right) \equiv \left( x^{p}+1\right)
f\left( \mathbf{F}_{\mathbf{x}}\right) -x^{p}f\left( \mathbf{F}_{\mathbf{x}%
}\right) =f\left( \mathbf{F}_{\mathbf{x}}\right) .
\end{equation*}%
For the particular cases, use Proposition \ref{P0}.
\end{proof}

\begin{corollary}
\label{P1}Let $n,r,s,m,q$ be non-negative integers and $p$ be a prime
number. Then, for any polynomials $f$ and $g$ with integer coefficients
there holds%
\begin{equation*}
\left( f\left( \mathbf{F}_{\mathbf{x}}\right) \right) ^{p} g\left( \mathbf{F}%
_{\mathbf{x}}\right)\equiv f\left( \mathbf{F}_{\mathbf{x}}\right)g\left( 
\mathbf{F}_{\mathbf{x}}\right) .
\end{equation*}
In particular, we have $\mathcal{F}_{mp+q}\left( x;r,s\right) \equiv 
\mathcal{F}_{m+q}\left( x;r,s\right) .$
\end{corollary}

\begin{proof}
By Fermat's little theorem and by twice application of Proposition \ref{C3}
we may state%
\begin{equation*}
LHS\equiv \underset{k=0}{\overset{p-1}{\sum }}\left( f\left( k\right)
\right) ^{p}g\left( k\right) \left( x+1\right) ^{p-1-k}x^{k}\equiv \underset{%
k=0}{\overset{p-1}{\sum }}f\left( k\right) g\left( k\right) \left(
x+1\right) ^{p-1-k}x^{k}
\end{equation*}%
which equals to the RHS.
\end{proof}

\begin{corollary}
\label{P21}For any non-negative integers $m\geq 1, n,r,s$ and any prime
number $p$, there hold%
\begin{align*}
\left( x+1\right) ^{s+1}\left( \mathcal{F}_{m\left( p-1\right) }\left(
x;r,s\right) -s!\right) & \equiv-\left( s-r^{\prime}\right) _{s}\left(
x+1\right) ^{r^{\prime}}x^{p-r^{\prime}},\text{ \ }r^{\prime}\neq0, \\
\left( x+1\right) ^{s+1}\left( \mathcal{F}_{m\left( p-1\right) }\left(
x;r,s\right) -s!\right) & \equiv-s!\left( x^{p}+1\right),\  \ r^{\prime}=0,
\end{align*}
where $r^{\prime}\equiv r$ and $r^{\prime}\in \left \{
0,1,\ldots,p-1\right
\} .$
\end{corollary}

\begin{proof}
Set $n=m\left( p-1\right) $ in the second particular case of Proposition \ref%
{C3}. \newline
If $r^{\prime }\neq 0$ we get%
\begin{align*}
\left( x+1\right) ^{s}\mathcal{F}_{m\left( p-1\right) }\left( x;r,s\right) &
\equiv \underset{k=0}{\overset{p-1}{\sum }}\left( r^{\prime }+k\right)
^{m\left( p-1\right) }\left( s+k\right) _{s}\left( x+1\right) ^{p-1-k}x^{k}
\\
& \equiv \underset{k=0,\  \ r^{\prime }+k\neq p}{\overset{p-1}{\sum }}\left(
s+k\right) _{s}\left( x+1\right) ^{p-1-k}x^{k} \\
& =\underset{k=0}{\overset{p-1}{\sum }}\left( s+k\right) _{s}\left(
x+1\right) ^{p-1-k}x^{k} \\
& \  \  \  \  \ -\left( s-r^{\prime }+p\right) _{s}\left( x+1\right) ^{r^{\prime
}-1}x^{p-r^{\prime }} \\
& \equiv \left( x+1\right) ^{s}\mathcal{F}_{0}\left( x;0,s\right) -\left(
s-r^{\prime }\right) _{s}\left( x+1\right) ^{r^{\prime }-1}x^{p-r^{\prime }}
\\
& \equiv s!\left( x+1\right) ^{s}-\left( s-r^{\prime }\right) _{s}\left(
x+1\right) ^{r^{\prime }-1}x^{p-r^{\prime }}
\end{align*}%
and if $r^{\prime }=0$ we get%
\begin{align*}
\left( x+1\right) ^{s+1}\mathcal{F}_{m\left( p-1\right) }\left( x;r,s\right)
& \equiv \underset{k=1}{\overset{p-1}{\sum }}\left( s+k\right) _{s}\left(
x+1\right) ^{p-k}x^{k} \\
& =\underset{k=0}{\overset{p-1}{\sum }}\left( s+k\right) _{s}\left(
x+1\right) ^{p-k}x^{k}-s!\left( x+1\right) ^{p} \\
& =\left( x+1\right) ^{s+1}\mathcal{F}_{0}\left( x;0,s\right) -s!\left(
x+1\right) ^{p} \\
& =s!\left( x+1\right) ^{s+1}-s!\left( x^{p}+1\right) .
\end{align*}%
which complete the proof.
\end{proof}

\noindent Now, we give some curious congruences on $(r,s)$-Fubini
polynomials and on $(r_{1},\ldots,r_{q})$-Fubini polynomials defined below.

\begin{theorem}
\label{T1}For any integers $n,m,r,s\geq 0$ and any prime number $p\nmid m,$
there holds%
\begin{equation*}
\underset{k=1}{\overset{p-1}{\sum }}\frac{\mathcal{F}_{n+k}\left(
x;r,s\right) }{\left( -m\right) ^{k}}\equiv \left( -m\right) ^{n}\left( 
\mathcal{F}_{p-1}\left( x;r+m,s\right) -s!\right).
\end{equation*}
\end{theorem}

\begin{proof}
Upon using the identity $x^{s}\mathcal{F}_{n}\left( x;r,s\right) =\left( 
\mathbf{F}_{\mathbf{x}}+r-s\right) ^{n}\left( \mathbf{F}_{\mathbf{x}}\right)
_{s}$ and the known congruence $\left( -m\right) ^{-k}\equiv \binom{p-1}{k}%
m^{p-1-k}$ we obtain%
\begin{eqnarray*}
x^{s}LHS &\equiv &\underset{k=0}{\overset{p-1}{\sum }}\binom{p-1}{k}%
m^{p-1-k}\left( \mathbf{F}_{\mathbf{x}}+r-s\right) ^{n+k}\left( \mathbf{F}_{%
\mathbf{x}}\right) _{s} \\
&=&\left( \mathbf{F}_{\mathbf{x}}+r-s\right) ^{n}\left( \mathbf{F}_{\mathbf{x%
}}+r+m-s\right) ^{p-1}\left( \mathbf{F}_{\mathbf{x}}\right) _{s} \\
&=&\underset{j=0}{\overset{n}{\sum }}\binom{n}{j}\left( -m\right)
^{n-j}\left( \mathbf{F}_{\mathbf{x}}+r+m-s\right) ^{j+p-1}\left( \mathbf{F}_{%
\mathbf{x}}\right) _{s} \\
&=&\left( -m\right) ^{n}\left( \mathbf{F}_{\mathbf{x}}+r+m-s\right)
^{p-1}\left( \mathbf{F}_{\mathbf{x}}\right) _{s} \\
&&+\delta _{\left( n\geq 1\right) }\underset{j=1}{\overset{n}{\sum }}\binom{n%
}{j}\left( -m\right) ^{n-j}\left( \mathbf{F}_{\mathbf{x}}+r+m-s\right)
^{j+p-1}\left( \mathbf{F}_{\mathbf{x}}\right) _{s} \\
&=&x^{s}\left( -m\right) ^{n}\mathcal{F}_{p-1}\left( x;r+m,s\right) \\
&&+\delta _{\left( n\geq 1\right) }x^{s}\underset{j=1}{\overset{n}{\sum }}%
\binom{n}{j}\left( -m\right) ^{n-j}\mathcal{F}_{p+j-1}\left( x;r+m,s\right)
\\
&\equiv &x^{s}\left( -m\right) ^{n}\mathcal{F}_{p-1}\left( x;r+m,s\right) \\
&& +\delta _{\left( n\geq 1\right) }x^{s}\underset{j=1}{\overset{n}{\sum }}%
\binom{n}{j}\left( -m\right) ^{n-j}\mathcal{F}_{j}\left( x;r+m,s\right) \\
&=&x^{s}\left( -m\right) ^{n}\mathcal{F}_{p-1}\left( x;r+m,s\right) +\delta
_{\left( n\geq 1\right) }x^{s}\left( \mathcal{F}_{n}\left( x;r,s\right)
-\left( -m\right) ^{n}s!\right) \\
&=&x^{s}\left[ \left( -m\right) ^{n}\mathcal{F}_{p-1}\left( x;r+m,s\right) +%
\mathcal{F}_{n}\left( x;r,s\right) -\left( -m\right) ^{n}s!\right] ,
\end{eqnarray*}%
where $\delta $ is the Kronecker's symbol, i.e. $\delta _{\left( n\geq
1\right) }=1$ if $n\geq 1$ and $0$ otherwise.
\end{proof}

\noindent Let $\mathbf{r}_{q}=\left( r_{1},\ldots ,r_{q}\right) $ be a
vector of non-negative integers and let%
\begin{equation*}
F_{n}\left( x;\mathbf{r}_{q}\right) =\underset{j=0}{\overset{n+\left \vert 
\mathbf{r}_{q-1}\right \vert }{\sum }}{n+\left \vert \mathbf{r}_{q}\right
\vert \brace j+r_{q}}_{\mathbf{r}_{q}}\left( j+r_{q}\right) !x^{j}, \  \
0\leq r_{1}\leq \cdots \leq r_{q},
\end{equation*}%
where ${n+\left \vert \mathbf{r}_{q}\right \vert \brace j+r_{q}}_{\mathbf{r}%
_{q}}$ are the $\left( r_{1},\ldots ,r_{q}\right) $-Stirling numbers defined
by Mihoubi et al. \cite{mih1}. This polynomial is a generalization of the $r$%
-Fubini polynomials $F_{n}\left( x;r\right) :=\mathcal{F}_{n}\left(
x;r,r\right) $.

\begin{proposition}
\label{P}For any non-negative integers $n,m$ and any prime $p\nmid m,$ there
holds%
\begin{equation*}
x^{r_{q}}\underset{k=1}{\overset{p-1}{\sum }}\frac{F_{n+k}\left( x;\mathbf{r}%
_{q}\right) }{\left( -m\right) ^{k}}\equiv \left( -m\right) ^{n}\left(
-m\right) _{r_{1}}\cdots \left( -m\right) _{r_{q}}\left( \mathcal{F}%
_{p-1}\left( x;m,0\right) -1\right) .
\end{equation*}%
In particular, for $q=1$ and $r_{q}=r$ we obtain 
\begin{equation*}
x^{r}\underset{k=1}{\overset{p-1}{\sum }}\frac{\mathcal{F}_{n+k}\left(
x;r,r\right) }{\left( -m\right) ^{k}}\equiv \left( -m\right) ^{n}\left(
-m\right) _{r}\left( \mathcal{F}_{p-1}\left( x;m,0\right) -1\right).
\end{equation*}
\end{proposition}

\begin{proof}
By the identity $\left(\mathbf{F}_{\mathbf{x}}\right)_{n}=n!x^{n}$ and by 
\cite[Th. 10]{mih1} we have%
\begin{eqnarray*}
x^{r_{q}}F_{n}\left( x;\mathbf{r}_{q}\right) &=&\underset{j=0}{\overset{%
n+\left \vert \mathbf{r}_{q-1}\right \vert }{\sum }}{n+\left \vert \mathbf{r}%
_{q}\right \vert \brace j+r_{q}}_{\mathbf{r}_{q}}\left( \mathbf{F}_{\mathbf{x%
}}\right) _{j+r_{q}} \\
&=&\underset{j=0}{\overset{n+\left \vert \mathbf{r}_{q-1}\right \vert }{\sum 
}}{n+\left \vert \mathbf{r}_{q}\right \vert \brace j+r_{q}}_{\mathbf{r}%
_{q}}\left( \mathbf{F}_{\mathbf{x}}-r_{q}\right) _{j}\left( \mathbf{F}_{%
\mathbf{x}}\right) _{r_{q}} \\
&=&\mathbf{F}_{\mathbf{x}}^{n}\left( \mathbf{F}_{\mathbf{x}}\right)
_{r_{1}}\cdots \left( \mathbf{F}_{\mathbf{x}}\right) _{r_{q}} \\
&=&\underset{k=0}{\overset{\left \vert \mathbf{r}_{q}\right \vert }{\sum }}%
a_{k}\left( \mathbf{r}_{q}\right) \mathbf{F}_{\mathbf{x}}^{n+k} \\
&=&\underset{j=0}{\overset{\left \vert \mathbf{r}_{q}\right \vert }{\sum }}%
a_{j}\left( \mathbf{r}_{q}\right) \mathcal{F}_{n+j}\left( x\right) ,
\end{eqnarray*}%
where $\underset{k=0}{\overset{\left \vert \mathbf{r}_{q}\right \vert }{\sum 
}}a_{k}\left( \mathbf{r}_{q}\right) u^{k}=\left( u\right) _{r_{1}}\cdots
\left( u\right) _{r_{q}}.$ So, by application of Theorem \ref{T1} we get 
\begin{eqnarray*}
x^{r_{q}}\underset{k=1}{\overset{p-1}{\sum }}\frac{F_{n+k}\left( x;\mathbf{r}%
_{q}\right) }{\left( -m\right) ^{k}} &=&\underset{k=1}{\overset{p-1}{\sum }}%
\frac{1}{\left( -m\right) ^{k}}\left( \underset{j=0}{\overset{\left \vert 
\mathbf{r}_{q}\right \vert }{\sum }}a_{j}\left( \mathbf{r}_{q}\right) 
\mathcal{F}_{n+k+j}\left( x;0,0\right) \right) \\
&=&\underset{j=0}{\overset{\left \vert \mathbf{r}_{q}\right \vert }{\sum }}%
a_{j}\left( \mathbf{r}_{q}\right) \underset{k=1}{\overset{p-1}{\sum }}\frac{%
\mathcal{F}_{n+j+k}\left( x;0,0\right) }{\left( -m\right) ^{k}} \\
&\equiv &\underset{j=0}{\overset{\left \vert \mathbf{r}_{q}\right \vert }{%
\sum }}a_{j}\left( \mathbf{r}_{q}\right) \left( -m\right) ^{n+j}\left( 
\mathcal{F}_{p-1}\left( x;m,0\right) -1\right) \\
&=&\left( -m\right) ^{n}\left( -m\right) _{r_{1}}\cdots \left( -m\right)
_{r_{q}}\left( \mathcal{F}_{p-1}\left( x;m,0\right) -1\right) . \ 
\end{eqnarray*}
\end{proof}

\begin{remark}
Since $x^{r_{q}}F_{n}\left( x;\mathbf{r}_{q}\right) =\mathbf{F}_{\mathbf{x}%
}^{n}\left( \mathbf{F}_{\mathbf{x}}\right) _{r_{1}}\cdots \left( \mathbf{F}_{%
\mathbf{x}}\right) _{r_{q}},$ then, for $f\left( x\right) =x^{mp+q}\left(
x\right) _{r_{1}}\cdots \left( x\right) _{r_{q}}$ and $g\left( x\right) =1$
in Corollary \ref{P1} we obtain%
\begin{equation*}
F_{mp+q}\left( x;\mathbf{r}_{q}\right) \equiv F_{m+q}\left( x;\mathbf{r}%
_{q}\right)
\end{equation*}%
and for $f\left( x\right) =x^{m\left( p-1\right) }\left( x\right)
_{r_{1}}\cdots \left( x\right) _{r_{q}}$ we get%
\begin{equation*}
F_{m\left( p-1\right) }\left( x;\mathbf{r}_{q}\right) \equiv F_{0}\left( x;%
\mathbf{r}_{q}\right) ,\  \ r_{1}\cdots r_{q}\neq 0,\ m\geq 0.
\end{equation*}
\end{remark}

\begin{corollary}
\label{CC} Let $a_{0}\left( x\right) ,\ldots ,a_{t}\left( x\right) $ be
polynomials with integer coefficients, 
\begin{equation*}
\mathcal{R}_{n,t}\left( x;r,s\right) =\underset{i=0}{\overset{t}{\sum }}%
a_{i}\left( x\right) \mathcal{F}_{n+i}\left( x;r,s\right) \text{ \ and \ }%
\mathcal{L}_{t}\left( x,y\right) =\underset{i=0}{\overset{t}{\sum }}%
a_{i}\left( x\right) y^{i}.
\end{equation*}%
Then, for any non-negative integers $n,m,r,s$ and any prime $p\nmid m,$
there hold%
\begin{equation*}
\underset{k=1}{\overset{p-1}{\sum }}\frac{\mathcal{R}_{n+k,t}\left(
x;r,s\right) }{\left( -m\right) ^{k}} \equiv \left( -m\right) ^{n}\mathcal{L}%
_{t}\left( x,-m\right) \left( \mathcal{F}_{p-1}\left(
x;r+m,s\right)-s!\right).
\end{equation*}
\end{corollary}

\begin{proof}
Theorem \ref{T1} implies%
\begin{eqnarray*}
\underset{k=1}{\overset{p-1}{\sum }}\frac{\mathcal{R}_{n+k,t}\left(
x;r,s\right) }{\left( -m\right) ^{k}} &=&\underset{j=0}{\overset{t}{\sum }}%
a_{j}\left( x\right) \underset{k=1}{\overset{p-1}{\sum }}\frac{\mathcal{F}%
_{n+k+j}\left( x;r,s\right) }{\left( -m\right) ^{k}} \\
&\equiv &\underset{j=0}{\overset{t}{\sum }}a_{j}\left( x\right) \left(
-m\right) ^{n+j}\left( \mathcal{F}_{p-1}\left( x;r+m,s\right) -s!\right) \\
&=&\left( -m\right) ^{n}\mathcal{L}_{t}\left( x,-m\right) \left( \mathcal{F}%
_{p-1}\left( x;r+m,s\right) -s!\right) . \  \  \  \ 
\end{eqnarray*}
\end{proof}

\section{Congruences involving $\mathcal{F}_{n}\left( x;r,s\right) ,$ $%
\mathcal{P}_{n}\left( x,r\right) $ and $\mathcal{T}_{n}\left( x,r\right) $}

\noindent The following theorem gives connection in congruences between the
polynomials $\mathcal{F}_{n}$ and $\mathcal{P}_{n}.$

\begin{theorem}
\label{T2}Let $n,r$ be non-negative integers and $p$ be a prime number. 
\newline
Then, for $m\in \left \{ 0,\ldots ,p-1\right \} $ there holds%
\begin{equation*}
\underset{k=m}{\overset{p-1}{\sum }}\left( -x\right) ^{k}\frac{\mathcal{F}%
_{n}\left( x;r+k,k\right) }{\left( k-m\right) !}\equiv \left( -1\right)
^{m}m!\left( r+m\right) ^{n}\mathcal{P}_{p-1}\left( x,m\right) .
\end{equation*}%
In particular, for $m=0,$ we get%
\begin{equation*}
\underset{k=0}{\overset{p-1}{\sum }}\left( -x\right) ^{k}\frac{\mathcal{F}%
_{n}\left( x;r+k,k\right) }{k!}\equiv r^{n}\left( 1+x+\cdots +x^{p-1}\right)
.
\end{equation*}
\end{theorem}

\begin{proof}
For $k<m$ we get $\left \langle m+1\right \rangle _{p-1-k}=0$ and for $m\leq
k\leq p-1$ we have%
\begin{equation*}
\left \langle m+1\right \rangle _{p-1-k}=\frac{\left( m+p-k-1\right) !}{m!}=%
\frac{\left( p-1-\left( k-m\right) \right) !}{m!}\equiv -\frac{1}{m!}\frac{%
\left( -1\right) ^{k-m}}{\left( k-m\right) !}.
\end{equation*}%
where $\left \langle x\right \rangle _{n}=x\left( x+1\right) \cdots \left(
x+n-1\right) $ if $n\geq 1$ and $\left \langle x\right \rangle _{0}=1.$ Then%
\begin{eqnarray*}
LHS &\equiv &-\left( -1\right) ^{m}m!\underset{k=0}{\overset{p-1}{\sum }}%
\left \langle m+1\right \rangle _{p-1-k}x^{k}\mathcal{F}_{n}\left(
x;r+k,k\right) \\
&\equiv &-\left( -1\right) ^{m}m!\underset{k=0}{\overset{p-1}{\sum }}\left
\langle m-p+1\right \rangle _{p-1-k}\left( \mathbf{F}_{\mathbf{x}}+r\right)
^{n}\left( \mathbf{F}_{\mathbf{x}}\right) _{k} \\
&\equiv &-\left( -1\right) ^{m}m!\underset{k=0}{\overset{p-1}{\sum }}\binom{%
p-1}{k}\left \langle m-p+1\right \rangle _{p-1-k}\left( \mathbf{F}_{\mathbf{x%
}}+r\right) ^{n}\left \langle -\mathbf{F}_{\mathbf{x}}\right \rangle _{k} \\
&=&-\left( -1\right) ^{m}m!\left \langle m-p+1-\mathbf{F}_{\mathbf{x}}\right
\rangle _{p-1}\left( \mathbf{F}_{\mathbf{x}}+r\right) ^{n} \\
&=&-\left( -1\right) ^{m}m!\left( \mathbf{F}_{\mathbf{x}}-m+p-1\right)
_{p-1}\left( \mathbf{F}_{\mathbf{x}}+r\right) ^{n} \\
&=&-\left( -1\right) ^{m}m!\left( \mathbf{F}_{\mathbf{x}}-m+r+m\right)
^{n}\left( \mathbf{F}_{\mathbf{x}}-m+p-1\right) _{p-1} \\
&=&-\left( -1\right) ^{m}m!\underset{j=0}{\overset{n}{\sum }}{n+r+m\brace %
j+r+m}_{r+m}\left( \mathbf{F}_{\mathbf{x}}-m\right) _{j}\left( \mathbf{F}_{%
\mathbf{x}}-m+p-1\right) _{p-1} .
\end{eqnarray*}%
But for $j\geq 1$ we have 
\begin{eqnarray*}
&&\left( \mathbf{F}_{\mathbf{x}}-m\right) _{j}\left( \mathbf{F}_{\mathbf{x}%
}-m+p-1\right) _{p-1} =\left( \mathbf{F}_{\mathbf{x}}-m+p-1\right) _{j+p-1}
\\
&\equiv& \left( \mathbf{F}_{\mathbf{x}}-m-1\right) _{j+p-1} =\left(
j+p-1\right) !\mathcal{P}_{j+p-1}\left( x,m+1\right)\equiv -\delta _{\left(
j=0\right)}\mathcal{P}_{p-1}\left( x,m+1\right),
\end{eqnarray*}
hence, it follows $LHS \equiv \left( -1\right) ^{m}m!\left( r+m\right) ^{n}%
\mathcal{P}_{p-1}\left( x,m\right)$.
\end{proof}

\noindent The following theorem gives connection in congruences between the
polynomials $\mathcal{F}_{n}$ and $\mathcal{T}_{n}$.

\begin{theorem}
\label{T3}For any integers $n,m,r\geq 0$ and any prime $p,$ there holds%
\begin{equation*}
\underset{k=0}{\overset{p-1}{\sum }}\left( -m\right) _{p-1-k}\left(
x+1\right) ^{k}\mathcal{F}_{n}\left( x;r+m,k\right) \equiv -r^{n}\mathcal{T}%
_{p-1}\left( x;m\right) .
\end{equation*}
\end{theorem}

\begin{proof}
Upon using the identity $\left( x+1\right) ^{s}\mathcal{F}_{n}\left(
x;r,s\right) =\left( \mathbf{F}_{\mathbf{x}}+r\right) ^{n}\left( \mathbf{F}_{%
\mathbf{x}}+s\right) _{s}$ and the known congruence $\left( m\right)
_{p-1-k}\equiv \binom{p-1}{k}\left \langle -m\right \rangle _{p-1-k}$ we
obtain 
\begin{eqnarray*}
LHS &=&\underset{k=0}{\overset{p-1}{\sum }}\left( -1\right) ^{k}\left
\langle m\right \rangle _{p-1-k}\left( x+1\right) ^{k}\mathcal{F}_{n}\left(
x;r+m,k\right) \\
&\equiv &\underset{k=0}{\overset{p-1}{\sum }}\binom{p-1}{k}\left \langle
m\right \rangle _{p-1-k}\left( \mathbf{F}_{\mathbf{x}}+r+m\right) ^{n}\left( 
\mathbf{F}_{\mathbf{x}}+k\right) _{k} \\
&\equiv &\underset{k=0}{\overset{p-1}{\sum }}\binom{p-1}{k}\left \langle
m\right \rangle _{p-1-k}\left( \mathbf{F}_{\mathbf{x}}+r+m\right) ^{n}\left
\langle \mathbf{F}_{\mathbf{x}}+1\right \rangle _{k} \\
&=&\left( \mathbf{F}_{\mathbf{x}}+r+m\right) ^{n}\left \langle \mathbf{F}_{%
\mathbf{x}}+m+1\right \rangle _{p-1} \\
&\equiv &\left( \mathbf{F}_{\mathbf{x}}+m+r\right) ^{n}\left( \mathbf{F}_{%
\mathbf{x}}+m+p-1\right) _{p-1} \\
&=&\underset{j=0}{\overset{n}{\sum }}{n+r\brace j+r}_{r}\left( \mathbf{F}_{%
\mathbf{x}}+m\right) _{j}\left( \mathbf{F}_{\mathbf{x}}+m+p-1\right) _{p-1}
\\
&=&\underset{j=0}{\overset{n}{\sum }}{n+r\brace j+r}_{r}\left( \mathbf{F}_{%
\mathbf{x}}+m+p-1\right) _{j+p-1} \\
&=&\underset{j=0}{\overset{n}{\sum }}{n+r\brace j+r}_{r}\left( j+p-1\right) !%
\mathcal{T}_{j+p-1}\left( x;m-j\right) \\
&=&\left( p-1\right) !\mathcal{T}_{p-1}\left( x;m\right) +\underset{j=1}{%
\overset{n}{\sum }}{n+r\brace j+r}_{r}\left( j+p-1\right) !\mathcal{T}%
_{j+p-1}\left( x;m-j\right) \\
&\equiv &-r^{n}\mathcal{T}_{p-1}\left( x;m\right) .
\end{eqnarray*}
\end{proof}

\noindent The following theorem gives connection in congruences between the
polynomials $\mathcal{T}_{n}$ and $\mathcal{P}_{n}.$

\begin{theorem}
\label{T4}For any integers $n,m,r\geq 0$ and any prime $p,$ there holds%
\begin{equation*}
\underset{k=0}{\overset{p-1}{\sum }}\left \langle m+r+1\right \rangle
_{p-1-k}\mathcal{T}_{n+k}\left( x;r\right) \equiv -\left( m+r+n\right) _{n}%
\mathcal{P}_{p-1}\left( x,m+1\right) .
\end{equation*}
\end{theorem}

\begin{proof}
Upon using the congruence $\left( m\right) _{p-1-k}\equiv \binom{p-1}{k}%
\left \langle -m\right \rangle _{p-1-k}$ and Proposition \ref{R1} we obtain%
\begin{eqnarray*}
LHS &\equiv &\underset{k=0}{\overset{p-1}{\sum }}\binom{p-1}{k}\left(
-r-m-1\right) _{p-1-k}\left( \mathbf{F}_{\mathbf{x}}+r+n+k\right) _{n+k} \\
&\equiv &\underset{k=0}{\overset{p-1}{\sum }}\binom{p-1}{k}\left(
-r-m-1\right) _{p-1-k}\left( \mathbf{F}_{\mathbf{x}}+r+n\right) _{n}\left( 
\mathbf{F}_{\mathbf{x}}+r\right) _{k} \\
&=&\left( \mathbf{F}_{\mathbf{x}}+r+n\right) _{n}\left( \mathbf{F}_{\mathbf{x%
}}-m-1\right) _{p-1} \\
&=&\underset{k=0}{\overset{n}{\sum }}\binom{n}{k}\left( r+n+m+p\right)
_{n-k}\left( \mathbf{F}_{\mathbf{x}}-m-p\right) _{k}\left( \mathbf{F}_{%
\mathbf{x}}-m-1\right) _{p-1} \\
&=&\underset{k=0}{\overset{n}{\sum }}\binom{n}{k}\left( r+n+m\right)
_{n-k}\left( \mathbf{F}_{\mathbf{x}}-m-1\right) _{k+p-1} \\
&\equiv &\underset{k=0}{\overset{n}{\sum }}\binom{n}{k}\left( r+n+m\right)
_{n-k}\left( k+p-1\right) !\mathcal{P}_{k+p-1}\left( x,m+1\right) \\
&\equiv &-\left( m+r+n\right) _{n}\mathcal{P}_{p-1}\left( x,m+1\right).
\end{eqnarray*}
\end{proof}

\begin{corollary}
Let $\mathcal{R}_{n,t}\left( x;r,s\right)$ be as in Corollary \ref{CC}.
Then, for any non-negative integers $n,m,r,s$ and any prime $p\nmid m,$
there holds%
\begin{equation*}
\underset{k=m}{\overset{p-1}{\sum }}\left( -x\right) ^{k}\dbinom{k}{m}\frac{%
\mathcal{R}_{n,t}\left( x;r+k,k\right) }{k!} \equiv \left( -1\right)
^{m}\left( r+m\right) ^{n}\mathcal{L}_{t}\left( x,r+m\right) \mathcal{P}%
_{p-1}\left( x,m\right) .
\end{equation*}
\end{corollary}

\begin{proof}
Theorem \ref{T2} implies%
\begin{eqnarray*}
LHS &=&\underset{j=0}{\overset{t}{\sum }}a_{j}\left( x\right) \underset{k=m}{%
\overset{p-1}{\sum }}\left( -x\right) ^{k}\dbinom{k}{m}\frac{\mathcal{F}%
_{n+j}\left( x;r+k,k\right) }{k!} \\
&\equiv &\underset{j=0}{\overset{t}{\sum }}a_{j}\left( x\right) \left(
-1\right) ^{m}\left( r+m\right) ^{n+j}\mathcal{P}_{p-1}\left( x,m\right) \\
&\equiv &\left( -1\right) ^{m}\left( r+m\right) ^{n}\mathcal{L}_{t}\left(
x,r+m\right) \mathcal{P}_{p-1}\left( x,m\right) .
\end{eqnarray*}
\end{proof}


\begin{thebibliography}{99}
\bibitem{ben} A. Benyattou and M. Mihoubi, Some applications of the
generalized Bell umbra in congruences. Les annales RECITS, 3 (2016) 35--46.
Avalaible electronically at: www.lrecits.usthb.dz

\bibitem{boy} K. N. Boyadzhiev, A series transformation formula and related
polynomials. \textit{Int. J. Math. Math. Sci.,} 2005 (2005) 3849-3866.

\bibitem{bro} A. Z. Broder, The $r$-Stirling numbers. \textit{Discrete Math.,%
} 49 (1984) 241--259.

\bibitem{das} M. E. Dasef and S.M. Kautz, Some sums of some importance. 
\textit{College Math. J.,} 28 (1997) 52-55.

\bibitem{di} Di Bucchianico and D.E. Loeb, A selected survey of umbral
calculus, \textit{Electron. J. Combin.,} Dynamic Surveys DS3 (2000).

\bibitem{dum} D. Dumont, Matrices d'Euler-Siedel, Seminaire lotharingien de
combinatorie. B05c, 1981.

\bibitem{fla} P. Flajolet and R. Sedgewick, \textit{\ Analytic combinatorics}%
. Cambridge Univ. Press (2009).

\bibitem{ges} I. M. Gessel, Applications of the classical umbral calculus. 
\textit{Algebra Universalis, } 49 (2003) 397--434.

\bibitem{jam} R. D. James, The factors of a square-free integer. \textit{%
Canad. Math. Bull., }11 (1968) 733--735.

\bibitem{mah} Mahir Bilen Can, Lex E. Renner, Ordered Bell numbers, Hermite
polynomials, skew Young tableaux, and Borel orbits. J. Comb. Theory, Ser. A
119 (2012) 1798--1810.

\bibitem{mez} I. Mez\H{o}, Periodicity of the last digits of some
combinatorial sequences. \textit{J. Integer Seq.,} 17 (2014) Article 14.1.1.

\bibitem{mih} M. Mihoubi, Bell polynomials and binomial type sequences. 
\textit{Discrete Math.,} 308 (2008) 2450--2459.

\bibitem{mih1} M. Mihoubi and M. S. Maamra, The $\left( r_{1},\ldots
,r_{p}\right) $-Stirling numbers of the second kind. \textit{Integers} 
\textbf{12} (2012), Article A35.

\bibitem{rob} T. J. Robinson, Formal calculus and umbral calculus. \textit{%
Electron. J. Combin.,} 17:\#R95 (2010).

\bibitem{rom} S. Roman, \textit{The umbral calculus.} Academic Press,
Orlando, FL (1984).

\bibitem{rom1} S. Roman. G. C. Rota, The umbral calculus. \textit{Adv. Math.,%
} 27 (1978) 95--188.

\bibitem{rota} G. C. Rota and B. D. Taylor, The classical umbral calculus. 
\textit{SIAM J. Math. Anal.} 25 (1994) 694--711.

\bibitem{sta} R. Stanley, \textit{Enumerative combinatorics, vol. 1.}
Cambridge Univ. Press, Cambridge (1997).

\bibitem{sun} Y. Sun, X. Wu and J. Zhuang, Congruences on the Bell
polynomials and the derangement polynomials.\textit{\ J. Number Theory, }
133 (2013) 1564--1571.

\bibitem{tan} S. M. Tanny, On some numbers related to the Bell numbers. 
\textit{Canad. Math. Bull., }17 (1975) 733--738.

\bibitem{vel} D. J. Velleman and G. S. Call, Permutations and combination
locks. \textit{Math. Mag.,} 68 (1995) 243--253.
\end{thebibliography}
\end{document}